\documentclass[11pt]{article}

\usepackage{epsf,epsfig,amsfonts,amsgen,amsmath,amstext,amsbsy,amsopn,amsthm}
\usepackage{color}
\usepackage{multirow}
\newtheorem{theorem}{Theorem}[section]
\newtheorem{prop}[theorem]{Proposition}
\newtheorem{lemma}[theorem]{Lemma}
\newtheorem{conjecture}[theorem]{Conjecture}

\newcounter{mathitem}
  {\begin{list}{{$(\roman{mathitem})$}}{
   \setcounter{mathitem}{0}
   \usecounter{mathitem}
   \setlength{\topsep}{0pt plus 2pt minus 0pt}
   \setlength{\parskip}{0pt plus 2pt minus 0pt}
   \setlength{\partopsep}{0pt plus 2pt minus 0pt}
   \setlength{\parsep}{0pt plus 2pt minus 0pt}
   \setlength{\leftmargin}{35pt}
   \setlength{\itemsep}{0pt plus 2pt minus 0pt}}}
  {\end{list}}

\begin{document}

\title{ The  Randi\'{c} index and signless Laplacian spectral radius  of graphs}
\author{Bo Ning
and  Xing Peng\footnote{Corresponding author.}\\[2mm]
\small Center for Applied Mathematics\\
\small Tianjin University, Tianjin,  300062 P. R. China\\
\small E-mail: bo.ning@tju.edu.cn; x2peng@tju.edu.cn}
\date{}
\maketitle

\begin{abstract}
 Given a connected graph $G$, the Randi\'c index $R(G)$ is  the sum of $\tfrac{1}{\sqrt{d(u)d(v)}}$ over all edges $\{u,v\}$ of $G$, where $d(u)$ and  $d(v)$ are the degree of  vertices $u$ and $v$ respectively. Let $q(G)$ be the largest eigenvalue of the singless Laplacian matrix of $G$ and  $n=|V(G)|$.
Hansen and Lucas (2010)   made the following conjecture:
\[
\frac{q(G)}{R(G)} \leq
\begin{cases}
  \frac{4n-4}{n} & 4 \leq n\leq 12\\
  \frac{n}{\sqrt{n-1}} & n\geq 13
   \end{cases}
\]
with equality if and only if $G=K_{n}$ for $4\leq n\leq 12$ and $G=S_n$ for $n\geq 13$, respectively. Deng, Balachandran, and Ayyaswamy ({\em J.~Math.~Anal.~Appl.} 2014)  verified this conjecture for $4 \leq n \leq 11$.  In this paper, we solve this conjecture completely.
\end{abstract}



\section{Introduction}
For a connected graph $G=(V,E)$, the {\it Randi{\'c} index}  $R(G)$ is  defined as
\[
R(G)=\sum_{\{u,v\} \in E(G)} \frac{1}{\sqrt{d(u)d(v)}},
\]
where $d(u)$  and $d(v)$ are the degree of vertices $u$ and $v$  respectively.   This parameter was introduced by the   chemist Milan Randi{\'c} \cite{Randic} in 1975 under the name `branching index'.  Originally, it was used  to  measure the extent of branching of the carbon-atom skeleton of saturated hydrocarbons.  It was   noticed that there is a good correlation between the Randi{\'c} index and several physico-chemical properties of alkanes: for example, boiling points, enthalpies of formation, chromatographic retention times, etc. \cite{HKM,KH,KHMR}.

From the view of extremal graph theory, one may ask what are the minimum and maximum values of the  Randi{\'c} index among a certain class of graphs and which graphs from the given class of graphs attain the extremal values.  Bollob\'as and   Erd\H{o}s \cite{BE}  first considered this kind of question. They proved that $R(G) \geq \sqrt{n-1}$ for each graph with $n$ vertices and without isolated vertices. Moreover, the equality holds if and only $G$ is the star. After that, there are a lot of references in  this vein, for example, \cite{BBG, BEA, DelFavRau,GaoLu}.   Bollob\'as, Erd\H{o}s, and Sarkar \cite{BEA} studied  generalizations of the Randi{\'c} index.

Another  direction of research is to ask the relationships between the  Randi{\'c} index and other parameters of graphs.  Hansen and Vukicevi\'c \cite{HansenV} studied the connections between
the Randi\'{c} index and the chromatic number of graphs.  Aouchiche, Hansen,  and Zheng \cite{AHZ}  made  a conjecture on the minimum values of $\tfrac{R(G)}{D(G)}$ and $R(G)-D(G)$ over all connected graphs with the same number of vertices, where $D(G)$ is the diameter of $G$.   Li and Shi \cite{LS}  as well as  Dvo{\v r}{\'a}k, Lidick{\'y}, and {\v S}krekovski \cite{DLS} studied this conjecture  before    Yang and Lu \cite{YangLu} finally  resolved it.  Another result is
$\lambda_1(G) \geq \tfrac{e(G)}{R(G)}$ which was proved by Favaron,  Mah\'eo, and Sacl\'e \cite{FMS}. Here $\lambda_1(G)$ is the largest eigenvalue of the adjacency matrix of $G$.  One may ask to prove  similar results involving  the  Randi{\'c} index and the spectral radius of other matrices associated with a graph.

For a graph $G$, the {\it singless Laplacian matrix} $Q$ is defined as $D+A$, where $D$ is the diagonal matrix of degrees in $G$ and $A$ is the adjacency matrix of $G$.  Let $q(G)$ be the largest eigenvalue of $Q$. With the aid of AutoGraphiX system, Hansen and Lucas \cite{HansenLucas} proposed the following two conjectures.  The first one is on the difference between $q(G)$ and $R(G)$. More precisely, they conjectured that  if  $G$ is a connected graph on $n \geq 4$ vertices, then $q(G)-R(G)\leq \tfrac{3n}{2}-2$ and equality holds for $G=K_n$.  This conjecture was proved by  Deng, Balachandran, and Ayyaswamy \cite{DengBalAyy}. The second one concerns the ratio of $q(G)$ to $R(G)$.
\begin{conjecture}[Hansen and Lucas \cite{HansenLucas}]\label{ConjHL2}
Let $G$ be a connected graph on $n\geq 4$ vertices with the largest signless Laplacian eigenvalue $q(G)$ and Randi\'c index $R(G)$.
Then
\[
\frac{q(G)}{R(G)} \leq
\begin{cases}
  \frac{4n-4}{n} & 4 \leq  n\leq 12\\
  \frac{n}{\sqrt{n-1}} & n\geq 13
   \end{cases}
\]
with equality if and only if $G=K_{n}$ for $4\leq n\leq 12$ and $G=S_n$ for $n\geq 13$, respectively.
\end{conjecture}
 Deng, Balachandran, and  Ayyaswamy \cite{DengBalAyy} were able to prove this conjecture for $4 \leq n \leq 11$ and established a nontrivial  upper bound on $\tfrac{q(G)}{R(G)}$ which is larger than the conjectured one.  We solve this conjecture completely in this paper. Namely, we prove the following theorem.
 \begin{theorem} \label{main}
 For a connected graph $G$ with $n$ vertices, we have
 \[
\frac{q(G)}{R(G)} \leq
\begin{cases}
  \frac{11}{3} & n= 12;\\
  \frac{n}{\sqrt{n-1}} & n\geq 13.
   \end{cases}
\]
The  equality holds if and only if  $G=K_{12}$ for $n=12$ and $G=S_n$ for $n\geq 13$.
  \end{theorem}
For  developments of the Randi\'c index, we refer interested  readers to excellent surveys, for instance,  Li and Gutman \cite{LiGutman}, Li and Shi \cite{LiShi},  as well as Li, Shi, and Wang \cite{LiShiWang}.

We follow the standard notation throughout this paper. For those not defined here, we refer the
reader to Bondy and Murty \cite{BoM}. For a graph $G=(V, E)$, the {\it neighborhood} $N_G(v)$ of a vertex $v$ is the set $\{u: u \in V(G) \textrm{ and } \{u,v\} \in E(G)\}$ and
 the {\it degree} $d_G(v)$ of a vertex $v$ is $|N_G(v)|$. If the graph $G$ is clear in the context, then we will drop the subscript $G$.   We will use $e(G)$ to denote the number of edges in $G$.

The paper is organized as follows. In Section 2, we will collect several previous results which are needed in the proof of the  main theorem.
 Also, we will  prove a number of  technical lemmas in Section 2.   The proof of Theorem \ref{main} will be given in Section 3.

\section{Preliminaries}
We first recall two theorems which provide upper bounds for the largest eigenvalue of the adjacency matrix and the signless Laplacian matrix of a graph respectively.
\begin{theorem}[Hong \cite{Hong}]\label{ThHong}
Let $G$ be a graph with $n$ vertices and $m$ edges. Let $\lambda_1$ be the largest eigenvalue of its adjacency matrix.
If the minimum degree $\delta(G)\geq 1$, then
\[
\lambda_1 \leq \sqrt{2m-n+1}.
\]
\end{theorem}

\begin{theorem}[Feng and Yu \cite{FengYu}] \label{FY}
Let $G$ be a graph with $n$ vertices and $m$ edges. If $q(G)$
is the singless Laplacian spectral radius of $G$,  then
\begin{align}
q(G)\leq \frac{2m}{n-1}+n-2.
\end{align}
\end{theorem}
For a vertex $v$ of a graph $G$, we define $m(v)$ as $\tfrac{1}{d(v)} \sum_{u \in N(v)} d(u)$. For a certain class of graphs, the following theorem gives a better upper bound on $q(G)$.
\begin{theorem}[Feng and Yu \cite{FengYu}]\label{ThMerris}
For a connected graph $G$, we have
 \[
 q(G) \leq \max\{d(v)+m(v): v\in V(G)\}.
 \]
\end{theorem}
We will need the following lemma.
\begin{lemma}\label{LeBasic}
For a connected graph $G$ with $n \geq 4$ vertices,  if $m=e(G) \geq n$ and  $R(G) > \sqrt{n-1}+\tfrac{2m-2n+2}{n\sqrt{n-1}}$, then $\tfrac{q(G)}{R(G)} < \tfrac{n}{\sqrt{n-1}}$.
\end{lemma}
\begin{proof}
Recall Theorem \ref{FY}.  We have
\[
\frac{q(G)}{R(G)} < \frac{\frac{2m}{n-1}+n-2}{ \sqrt{n-1}+\frac{2m-2n+2}{n\sqrt{n-1}}}.
\]
We note
\begin{align*}
\left(\sqrt{n-1}+\frac{2m-2n+2}{n\sqrt{n-1}}\right)\frac{n}{\sqrt{n-1}}&=\sqrt{n-1} \cdot  \frac{n}{\sqrt{n-1}}+  \frac{2m-2n+2}{n\sqrt{n-1}} \cdot  \frac{n}{\sqrt{n-1}}\\
                                                                                                                             &=n+\frac{2m-2n+2}{n-1}\\
                                                                                                                              &= \frac{2m}{n-1}+n-2.
\end{align*}
 This lemma follows easily.
\end{proof}
We recall the following lower bound for $R(G)$.
\begin{theorem}[Bollob\'{a}s and Erd\H{o}s \cite{BE}]\label{ThBE}
Let $G$ be a graph with $n$ vertices. If $\delta(G)\geq 1$,
then $R(G)\geq \sqrt{n-1}$ and the equality holds if and only if $G=S_n$.
\end{theorem}
If $\delta(G) \geq 2$, then we need the following better lower bound for $R(G)$.
\begin{theorem}[Delorme, Favaron,  and Rautenbach \cite{DelFavRau}]\label{ThDFR}
Let $G$ be a graph on $n$ vertices. If $\delta(G)\geq 2$,
then
\begin{align}
R(G)\geq \sqrt{2(n-1)}+\frac{1}{n-1}-\sqrt{\frac{2}{n-1}}.
\end{align}
\end{theorem}
A consequence of the theorem above is the following lemma.
\begin{lemma}\label{LeMin2}
Let $G$ be a connected graph with $n\geq 12$ vertices
and $n+k$ edges, where $1\leq k\leq 10$. If $\delta(G) \geq 2$, then
\[
R(G) > \sqrt{n-1}+\frac{2(k+1)}{n\sqrt{n-1}}.
\]
\end{lemma}

\begin{proof}
Recall that $n \geq 12$ and $e(G)=n+k$, where $1 \leq k \leq 10$.
By Theorem \ref{ThDFR}, we have
\[
R(G)\geq \frac{2n-4}{\sqrt{2n-2}}+\frac{1}{n-1}.
\]
When $n \geq 9$,  we   can verify
\[
\frac{2n-4}{\sqrt{2n-2}}+\frac{1}{n-1} > \sqrt{n-1}+\frac{2(k+1)}{n\sqrt{n-1}}
\]
easily.
\end{proof}
Among all unicyclic graphs, the minimum value of $R(G)$ is also known.
\begin{theorem}[Gao and Lu \cite{GaoLu}]\label{ThGL}
Let $G$ be a unicyclic graph on $n$ vertices. Then $R(G)$ attains its minimum value
when $G$ is  $S^{\ast}_n$, where $S^{\ast}_n$ is  obtained from the  star with $n$  vertices by adding an edge between leaves.
\end{theorem}
The following theorem allows us to  compare the Randi\'c index of a  graph and a  related graph obtained by deleting a minimum degree vertex.
\begin{theorem}[Hansen and Vukicevi\'c \cite{HansenV}]\label{ThHV}
Let $G$ be a graph with the Randi\'{c} index $R$, minimum degree $\delta$ and maximum
degree $\Delta$.  If  $v$ is a vertex of $G$ with degree  $\delta$,  then
\[
R(G)-R(G-v) \geq \frac{1}{2}\sqrt{\frac{\delta}{\Delta}}.
\]
\end{theorem}
The following lemma will be useful for us later.
\begin{lemma}\label{LeInduBas}
Let $G$ be a connected graph with $n$ vertices and $e(G)=n+k$, where $1\leq k \leq 10$. Let $v$
be a vertex with $d(v)=1$. If
\[
R(G-v) > \sqrt{n-2}+\frac{2(k+1)}{(n-1)\sqrt{n-2}},
\]
then we have
\[
R(G) > \sqrt{n-1}+\frac{2(k+1)}{n\sqrt{n-1}}.
\]
\end{lemma}
\begin{proof}
Let $v_0$ be a vertex
with $d(v_0)=\Delta$.  Recall  $v$ is  a vertex with $d(v)=1$ by the assumption. Let $H$ be the subgraph induced by $V(G)-\{v_0,v\}$ and
\[
L:=\sum_{\{x,y\} \in E(H)}\frac{1}{\sqrt{d_G(x)d_G(y)}}.
\]
If $\Delta=n-1$, then observe that
\[
R(G-v)=\left(R(G)-\frac{1}{\sqrt{n-1}}-L \right)\cdot \frac{\sqrt{n-1}}{\sqrt{n-2}}+L.
\]
Thus,
\begin{align*}
R(G)&=L+\frac{1}{\sqrt{n-1}}+\frac{\sqrt{n-2}}{\sqrt{n-1}}(R(G-v)-L)\\
&=\frac{\sqrt{n-2}}{\sqrt{n-1}}R(G-v)+\frac{1}{\sqrt{n-1}}+\left(1-\frac{\sqrt{n-2}}{\sqrt{n-1}}\right)L\\
&>\frac{\sqrt{n-2}}{\sqrt{n-1}}\left(\sqrt{n-2}+\frac{2(k+1)}{(n-1)\sqrt{n-2}}\right)+\frac{1}{\sqrt{n-1}}\\
&=\sqrt{n-1}+\frac{2(k+1)}{(n-1)\sqrt{n-1}}\\
&>\sqrt{n-1}+\frac{2(k+1)}{n\sqrt{n-1}}.
\end{align*}
If $\Delta\leq n-2$, by Theorem \ref{ThHV}, we have
$R(G)\geq R(G-v)+\frac{1}{2}\sqrt{\frac{1}{n-2}}$.
Thus
\[
R(G)\geq \sqrt{n-2}+\frac{2(k+1)}{(n-1)\sqrt{n-2}}+\frac{1}{2}\sqrt{\frac{1}{n-2}}>\sqrt{n-1}+\frac{2(k+1)}{n\sqrt{n-1}}.
\]
The proof is complete.
\end{proof}
We need the following proposition involving  the  vertex deletion.
\begin{prop} \label{delete}
For a connected graph $G$, assume $(v_1,\ldots,v_s)$ is an ordered set of vertices.   Let $G_0=G$ and $G_{i}=G_{i-1}-v_i$ for $1 \leq i \leq s$. If $v_{i}$ has degree one in $G_{i-1}$ for each $1 \leq i \leq s$, then  we have
\[
R(G) \geq \sum_{i=0}^{s-1} \frac{1}{2\sqrt{\Delta(G_{i})}}+R(G_s).
\]
\end{prop}
\begin{proof}
 Since we assume for each $1 \leq i \leq s$, the vertex $v_i$ has degree one in $G_{i-1}$. If we delete the vertex $v_{i}$ from $G_{i-1}$, then we have $R(G_{i-1}) \geq \tfrac{1}{2\sqrt{\Delta(G_{i-1})}}+R(G_{i})$ by Theorem \ref{ThHV}.  Since this observation holds for all $1 \leq i \leq s$,  the proposition follows.
\end{proof}
Lastly, we  need the following theorem.
\begin{theorem}[Favaron, Mah\'eo, and Sacl\'{e} \cite{FMS}]\label{Th}
For any connected graph $G$ with $m$ edges. If $R$ is the Randi\'c index and $\lambda_1$ is the largest eigenvalue of its adjacency matrix, then
$\lambda_1 \geq \tfrac{m}{R}$.
\end{theorem}

%

\section{Proof of the Main Theorem}
The following lemma is the key ingredient in the course of proving the main theorem.
\begin{lemma} \label{reduce}
Let $G$ be a connected graph with $n$ vertices and $m$ edges.
If $n\geq 15$ and $n+8\leq m\leq \min\{ 2n^{3/2}, \binom{n}{2}\}$, then
\[
\frac{q(G)}{R(G)} < \frac{n}{\sqrt{n-1}}.
\]
\end{lemma}
\begin{proof}
We note $2n^{3/2}> \binom{n}{2}$ when $15 \leq n \leq 17$ and $2n^{3/2}< \binom{n}{2}$  when $n \geq 18$. We define a function
\[
f(m)=\frac{m}{\sqrt{2m-(n-1)}}-\sqrt{n-1}-\frac{2m-2(n-1)}{n\sqrt{n-1}}.
\]
With the help of computer, one can check $f(m)>0$ for $15 \leq n \leq 17$ and $n+8 \leq m \leq \binom{n}{2}$. We assume $n \geq 18$ for the rest of the proof and  $\min\{ 2n^{3/2}, \binom{n}{2}\}=2n^{3/2}$ in this case.
We also consider a relevant  function
\[
g(m)=mn\sqrt{n-1}-n(n-1)\sqrt{2m-(n-1)}-(2m-2(n-1))\sqrt{2m-(n-1)}.
\]
To show $f(m) >0$,  it suffices to show $g(m) > 0$ for $ n+8 \leq m \leq 2n^{3/2}$ and $n \geq 18$ as $2m-(n-1)>0$.   Let $A=mn\sqrt{n-1}$ and $B=n(n-1)\sqrt{2m-(n-1)}+(2m-2(n-1))\sqrt{2m-(n-1)}=(2m+n^2-3n+2) \sqrt{2m-(n-1)}$. We define
\[
h(m)=A^2-B^2=m^2n^2(n-1)-(2m+n^2-3n+2)^2(2m-(n-1)).
\]
It is equivalent to prove $h(m)>0$ for  $n+8 \leq m \leq 2n^{3/2}$ and $n \geq 18$.  We first show $h(n+8)>0$ for $n \geq 18$.  We note
\[
h(n+8)=45n^3-657n^2+288n-5508.
\]
We can show $45n^3-657n^2+288n-5508 >0$ when $n=18$ directly.
By  taking the derivative, we can prove that  $45n^3-657n^2+288n-5508$ is increasing when $n \geq 18$, which completes the proof of   $h(n+8)>0$ for all $n \geq 18$.

We next show for fixed $n \geq 18$, the function $h(m)$ is increasing when $n+8 \leq m \leq 2n^{3/2}$.
 The derivative of $h(m)$  satisfies
\[
h'(m)=2mn^2(n-1)-4(n^2+2m-3n+2)(2m-n+1)-2(n^2+2m-3n+2)^2.
\]
It is enough to show $h'(m)>0$ for $n+8 \leq m \leq 2n^{3/2}$ and $n \geq 18$.  Let $l(m)=h'(m)$.   Taking derivative, we have
 \[
 l'(m)=2n^3-18n^2-48m+56n-40.
 \]
Also, the second derivative   $l''(m)=-48$. Therefore, the function $l(m)$ is concave down. If we can show $l(n+8)>0$ and $l(2n^{3/2})>0$ for $n \geq 18$,  then we establish $l(m)>0$ for all $n+8 \leq m \leq 2n^{3/2}$.
  We notice $l(n+8)=14n^3-154n^2+68n-1872>0$ when $n \geq 18$. We get
  \[
  l(2n^{3/2})=4n^{9/2}-2n^4-36n^{7/2}-80n^3+112n^{5/2}-42n^2-80n^{3/2}+44n-16.
  \]
One can confirm $l(2n^{3/2})>0$ for $n=18$ and $l(2n^{3/2})$ is increasing when $n \geq 18$ easily by taking derivative.  We already proved $l(m)=h'(m)>0$ when $n+8 \leq m \leq 2n^{3/2}$ and $n \geq 18$.  Combining with $h(n+8)>0$, we get $h(m)>0$  when $n+8 \leq m \leq 2n^{3/2}$ and $n \geq 18$. Thus, $f(m)>0$  for $n\geq 15$ and $n+8\leq m\leq \min\{ 2n^{3/2}, \binom{n}{2}\}$, that is,
\[
\frac{m}{\sqrt{2m-(n-1)}}>\sqrt{n-1}-\frac{2m-2(n-1)}{n\sqrt{n-1}}.
\]
By Theorems \ref{Th} and \ref{ThHong}, we have
\[
R(G) \geq \frac{m}{\lambda_1} \geq \frac{m}{\sqrt{2m-(n-1)}} > \sqrt{n-1}+\frac{2m-2(n-1)}{n\sqrt{n-1}}.
\]
By Lemma \ref{LeBasic}, we get
\[
\frac{q(G)}{R(G)} < \frac{n}{\sqrt{n-1}}.
\]
\end{proof}
Similar  to the lemma above, we can  prove  the following one.
\begin{lemma} \label{1314}
Let $G$ be a connected graph with $n$ vertices and $m$ edges.
If $n=13$ and $24\leq m\leq 78=\binom{13}{2}$, or $n=14$ and $23\leq m\leq 91=\binom{14}{2}$, then
\[
\frac{q(G)}{R(G)} < \frac{n}{\sqrt{n-1}}.
\]
\end{lemma}
We can show $f(m)>0$ using computer very easily, which is sufficient to prove the lemma by noticing Theorems \ref{Th}, \ref{ThHong}, and \ref{LeBasic}.

If a graph is relatively dense, then we can show the desired upper bound for $\tfrac{q(G)}{R(G)}$ easily  by  the following lemma.
\begin{lemma} \label{dense}
Let $G$ be a connected graph with $n$ vertices and $m$ edges.
If $m\geq  2n^{3/2}$, then
\[
\frac{q(G)}{R(G)} < \frac{n}{\sqrt{n-1}}.
\]
\end{lemma}
\begin{proof}
If $m\geq 2n^{3/2}$, then by the definition of $R(G)$, we have
\begin{align}
R(G)=\sum_{xy\in E(G)}\frac{1}{\sqrt{d(x)d(y)}}\geq \frac{m}{n-1}\geq \frac{2n^{3/2}}{n-1}.
\end{align}
Recall the well-known fact  $q(G)\leq 2\Delta\leq 2(n-1)$. Thus, we have
\[
\frac{q(G)}{R(G)}\leq\frac{(n-1)^2}{n^{3/2}} < \frac{n}{\sqrt{n-1}}.
\]
\end{proof}
In the case of  graphs with small maximum degree, the following lemma will prove the main theorem.
\begin{lemma}\label{small}
Let $G$ be a connected graph with $n$ vertices. If $\Delta(G) < n/2$, then we have
\[
\frac{q(G)}{R(G)} < \frac{n}{\sqrt{n-1}}.
\]
\end{lemma}
\begin{proof}
We note $q(G) \leq 2\Delta <n$ and   $R(G)\geq \sqrt{n-1}$ by Theorem \ref{ThBE}.  We get $\frac{q(G)}{R(G)} < \frac{n}{\sqrt{n-1}}$.
\end{proof}
With strong assumptions on the maximum degree and the number of edges of a graph, we are able to establish the desired upper bound on $\tfrac{q(G)}{R(G)}$.
\begin{lemma} \label{largedegree}
Let $G$ be a connected graph with $n \geq 13$ vertices. If either of the following cases holds:\\
(1) $n/2\leq \Delta\leq n-4$ and $e(G)=n+k$ for $1\leq k\leq 10$;\\
(2) $\Delta=n-3$ and $e(G)=n+k$ for $1\leq k\leq 7$;\\
(3) $\Delta=n-2$ and $e(G)=n+k$ for $1\leq k\leq 4$,\\
then
\[
\frac{q(G)}{R(G)} < \frac{n}{\sqrt{n-1}}.
\]
\end{lemma}

\begin{proof}
We use Theorem \ref{ThMerris} to show $q(G) \leq n$.
 For any vertex $v\in V(G)$, let $S_v=V(G) \setminus (\{v\} \cup N(v))$.
We shall show
\[
t(v):=d(v)+m(v) \leq n
\]
for each $v\in V(G)$.  We note $2(n+k)=\sum_{w \in V(G)} d(w)=d_v+\sum_{w \in N(v)} d_w+\sum_{w \in S_v} d_w$.
Therefore,  we have
\begin{align*}
t(v) & = d(v)+\frac{2(n+k)-d(v)-\sum_{u\in S_v}d(u)}{d(v)}\\
&\leq d(v)+\frac{2(n+k)-d(v)-(n-1-d(v))}{d(v)}\\
&=d(v)+\frac{n+2k+1}{d(v)}.
\end{align*}
Consider the function $f(x)=x+\frac{n+2k+1}{x}$.
We know $f(x)$ is increasing when $x\in (\sqrt{n+2k+1},\infty)$
and decreasing when $x\in (1,\sqrt{n+2k+1})$. Furthermore,
for any vertex $v$ with degree at least 4, we have
\[
t(v)\leq \max \left\{\Delta+\frac{n+2k+1}{\Delta},4+\frac{n+2k+1}{4}\right\}\leq\max\left\{\Delta+\frac{n+2k+1}{\Delta},n\right\}.
\]
Here we note $4+\tfrac{n+2k+1}{4} \leq n$ when $1 \leq k \leq 10$ and $n \geq 13$.
Suppose  (1) holds.
When $n\geq 13$ and $k\leq 10$, we have
\[
(n-4)\geq \Delta \geq n/2 > \sqrt{n+21}\geq \sqrt{n+2k+1}.
\]
Thus
\[
\Delta+\frac{n+2k+1}{\Delta}\leq (n-4)+\frac{n+21}{n-4} <  n
\]
when $n\geq 13$.
If (2) holds, then we get
\[
\Delta+\frac{n+2k+1}{\Delta}\leq (n-3)+\frac{n+15}{n-3} < n
\]
when $n\geq 13$. If (3) holds, then we obtain
\[
\Delta+\frac{n+2k+1}{\Delta}\leq (n-2)+\frac{n+9}{n-2}\leq n
\]
when $n\geq 13$.

Now we  need only to consider the vertices with degree 1, or 2, or 3. If $d(v)=1$,
then $t(v)=d(v)+m(v)\leq 1+\Delta\leq n$. If $d(v)=2$, then $t(v)\leq 2+\Delta\leq 2+(n-2)=n$.
If $d(v)=3$ and $\Delta\leq n-3$, then $t(v)\leq 3+(n-3)=n$.  We are left with $d(v)=3$ and $\Delta=n-2$,
In this case, $k\leq 4$. Therefore,  $t(v)\leq 3+\frac{n+9}{3}\leq n$ when $n\geq 13$.

By Theorem \ref{ThBE}, we have $R(G) > \sqrt{n-1}$ if $G$ is connected and $e(G) \geq n$. Thus, $\frac{q(G)}{R(G)} < \frac{n}{\sqrt{n-1}}$.
\end{proof}
We need the following lemma for the case of $n=12$.
\begin{lemma} \label{n121}
Let $G$ be a connected graph with $12$ vertices. If either of the following cases holds:\\
(1) $6\leq \Delta(G)\leq 8$ and $e(G)=12+k$ for $1\leq k\leq 8$;\\
(2) $\Delta(G)=9$ and $e(G)=12+k$ for $1\leq k\leq 6$;\\
(3) $\Delta(G)=10$ and $e(G)=12+k$ for $1\leq k\leq 3$,\\
then
\[
\frac{q(G)}{R(G)} < \frac{12}{\sqrt{11}}.
\]
\end{lemma}
The proof of the lemma is exactly the same as the one for proving Lemma \ref{largedegree} and it is omitted here.

The next three lemmas will deal with those graphs with large maximum degree and  small number of edges.
\begin{lemma} \label{base13}
Let $G$ be a connected graph with $ 13$ vertices. If either of the following holds:
\begin{enumerate}
\item $\Delta(G)=12$ and $e(G)=13+k$ for $1\leq k\leq  10$;
\item $\Delta(G)=11$ and $e(G)=13+k$ for $ 5 \leq k \leq  10$;
\item $\Delta(G)=10$, and $e(G)=13+k$ for $8 \leq k \leq 10$,
\end{enumerate}
then we have
\[
R(G) > \sqrt{12}+\frac{2(k+1)}{13\sqrt{12}}.
\]
\end{lemma}
\begin{proof}
Since proofs of three cases are very similar, we will present the detailed proof of Case 1 and sketch  proofs of others.  For each case, we will assume $v_0$ is a vertex with the maximum degree and $N_G(v_0)=\{v_1,\ldots,v_{\Delta}\}$.  If $\delta(G) \geq 2$, then  Lemma \ref{LeMin2} will complete the proof. Thus, we assume $G$ has at least one vertex with degree one in each case.

\vspace{0.1cm}
\noindent
{\bf  Case 1:} $\Delta(G)=12$.  We first consider the case of   $k=1$, i.e., $e(G)=14$.  Let $H$ be the subgraph induced by $N_G(v_0)$. We have $H$  is either a $P_3$ together with 9 isolated vertices or
two disjoint edges together with 8 isolated vertices. For the former case,  we have
\[
R(G)=\frac{9}{\sqrt{12}}+\frac{2}{\sqrt{2\cdot 12}}+\frac{1}{\sqrt{3\cdot 12}}+\frac{2}{\sqrt{2\cdot 3}}>\sqrt{12}+\frac{4}{13\sqrt{12}}.
\]
For the latter case, we have
\[
R(G)=\frac{8}{\sqrt{12}}+\frac{4}{\sqrt{2\cdot 12}}+\frac{2}{\sqrt{2\cdot 2}}>\sqrt{12}+\frac{4}{13\sqrt{12}}.
\]
Next assume  $2\leq k\leq 10$.   Recall that $d(v_0)=12$ and $N_G(v_0)=\{v_1,\ldots,v_{12}\}$. Let  $\{v_1,v_2,\ldots,v_s\}$ be the set of vertices with degree one in $G$.

 When $k=10$,  we claim $1 \leq s \leq 6$.  Otherwise, $s \geq 7$.    Let $G'$ be the subgraph induced by $\{v_{s+1},\ldots,v_{12}\}$. We have $e(G') =e(G)-d(v_0)= 11$. Since $s \geq 7$, we have $|V(G')| \leq 5$.  However, $G'$ can have  at most $\binom{5}{2}=10$ edges, which is a contradiction.
 Repeating the argument above, we can show $s \leq 7 $ when   $6 \leq k \leq 9$. Similarly, we have $s \leq 8$ when $3\leq k \leq 5$.  In the case of  $k=2$, we have $s \leq 9$.

 We next apply Proposition \ref{delete} with $(v_1,\ldots,v_s)$. We observe that  $d_{G_i}(v_{i+1})=1$ and $\Delta(G_i)=12-i$ for $0 \leq i \leq s-1$. Moreover, $|V(G_s)|=13-s$  and  $\delta(G_s) \geq 2$.  Recalling Theorem \ref{ThDFR}, we have
 \begin{equation} \label{lb13}
 R(G) \geq \sum_{i=0}^{s-1} \frac{1}{2\sqrt{12-i}}+\sqrt{2(12-s)}+\frac{1}{12-s}-\sqrt{\frac{2}{12-s}}.
 \end{equation}
Since we have proved an upper bound on $s$ depending on the value of $k$,  the inequality
 \begin{equation}\label{lb131}
 \sum_{i=0}^{s-1} \frac{1}{2\sqrt{12-i}}+\sqrt{2(12-s)}+\frac{1}{12-s}-\sqrt{\frac{2}{12-s}} > \sqrt{12}+\frac{2(k+1)}{13\sqrt{12}}
 \end{equation}
  can be verified using  the computer for each  $k$.

\vspace{0.1cm}
 \noindent
 {\bf Case 2:}  $\Delta(G)=11$. Let $\{v_{12}\}=V(G) \setminus (\{v_0\} \cup N(v_0))$.  We have two subcases depending on the degree of $v_{12}$.

  {\bf Subcase 2.1:} $d(v_{12})=1$. Let $G_1=G-v_{12}$. If $\{v_1,\ldots,v_t\}$ is the set of vertices of degree one in $G_1$, then we can prove an upper bound on $s=t+1$ depending on the value of $k$ by the same argument as Case 1.   We apply Proposition \ref{delete} with $(v_{12},v_1,\ldots,v_{s-1})$. We  observe  $\Delta(G_i) \leq 12-i$ for $0 \leq i \leq s-1$, $|V(G_s)|=13-s$, and  $\delta(G_s) \geq 2$.  Therefore, inequalities \eqref{lb13} and \eqref{lb131} still hold for this case and we can prove the desired lower bound for $R(G)$ similarly.

 {\bf  Subcase 2.2:}  $d(v_{12}) \geq 2$.  Let $\{v_1,\ldots,v_s\}$ be the set of  vertices with degree one in $G$.  Repeating the argument for Case 1,  we can get the asserted lower bound on $R(G)$.
 Here, we note $\Delta(G_i) \leq 12-i$ for $0 \leq i \leq s-1$ still holds when we apply Proposition \ref{delete}.   We may have a smaller upper bound on $s$ than the one in Case 1 for the same value of $k$, which does not affect the result.

 \vspace{0.1cm}
\noindent
{\bf Case 3:}  $\Delta(G)=10$.   Let $\{v_{11}, v_{12}\}=V(G) \setminus (\{v_0\} \cup N(v_0))$.  We have two subcases.

  {\bf Subcase 3.1:} $d(v_{11}),d(v_{12}) \geq 2$.  Let $\{v_1,\ldots,v_s\}$ be the set of  vertices with degree one  in  $G$.   We  can repeat the argument in Subcase 2.2  to show the desired lower bound for $R(G)$.

  {\bf Subcase 3.2:}  Either $d(v_{11})=1$ or $d(v_{12}) =1$.  We assume $d(v_{11})=1$. Let $G_1=G-v_{11}$.

   If $d_{G_1}(v_{12})=1$, then we define $G_2=G_1-v_{12}$. Let $\{v_1,\ldots,v_t\}$ be the set of vertices with degree one  in  $G_2$. We can use the argument in Case 1 to show an upper bound on $s+2$ depending on the value of $k$. We apply Proposition \ref{delete} with $(v_{11},v_{12},v_1,\ldots,v_t)$. We still have   $\Delta(G_i) \leq 12-i$ for $0 \leq i \leq s-1$. Therefore,  inequalities \eqref{lb13} and \eqref{lb131}  are true and the claimed lower bound for $R(G)$ follows.

   If $d_{G_1}(v_{12}) \geq 2$,  then we can use the argument for Subcase 2.1 to complete the proof of this lemma.
\end{proof}
We will need the following lemma for $n=12$.
\begin{lemma} \label{n122}
Let $G$ be a connected graph with $ 12$ vertices. If either of the following holds:
\begin{enumerate}
\item $\Delta(G)=11$ and $e(G)=12+k$ for $1\leq k\leq  8$;
\item $\Delta(G)=10$ and $e(G)=12+k$ for $ 4 \leq k \leq  8$;
\item $\Delta(G)=9$, and $e(G)=12+k$ for $7 \leq k \leq 8$,
\end{enumerate}
then we have
\[
R(G) >  \sqrt{11}+\frac{2(k+1)}{12\sqrt{11}}.
\]
\end{lemma}
We skip the proof here because it uses the same argument as the proof of Lemma \ref{base13}.

The next lemma  is in the same sprit of Lemma \ref{base13}.
\begin{lemma} \label{general}
Let $G$ be a connected graph with $n \geq 13$ vertices. If either of the following holds:
\begin{enumerate}
\item $\Delta(G)=n-1$ and $e(G)=n+k$ for $1\leq k\leq  10$;
\item $\Delta(G)=n-2$ and $e(G)=n+k$ for $ 5 \leq k \leq  10$;
\item $\Delta(G)=n-3$, and $e(G)=n+k$ for $8 \leq k \leq 10$,
\end{enumerate}
then we have
\[
R(G) >  \sqrt{n-1}+\frac{2(k+1)}{n\sqrt{n-1}}.
\]
\end{lemma}
\begin{proof}
We prove the lemma by induction on $n$.  The base case $n=13$ is given by Lemma \ref{base13}.   We assume the lemma holds for $|V(G)|=n$. For the inductive step where  $|V(G)|=n+1$ ,  if $\delta(G) \geq 2$, then the lemma follows from Theorem \ref{LeMin2}. Thus we assume $G$ has at least one vertex with degree one.  We assume further $v_0$ is a vertex with maximum degree. We  have three cases.

\vspace{0.1cm}
\noindent
{\bf Case 1:} $\Delta(G)=|V(G)|-1=n$.  Let $v_1 \in N_G(v_0)$ be a vertex with degree one.  If we define $G'=G-v_1$, then we have $|V(G')|=n$ and $\Delta(G')=|V(G)'|-1$. We have
$R(G')  > \sqrt{n-1}+\tfrac{2(k+1)}{n\sqrt{n-1}}$ by the inductive hypothesis.    Lemma \ref{LeInduBas} completes the proof of this case.

\vspace{0.1cm}
\noindent
{\bf Case 2:} $\Delta(G)=|V(G)|-2=n-1$.  Assume $\{v_n\}=V(G) \setminus (\{v_0\} \cup N(v_0))$.  If $d(v_n)=1$, then we let $G'=G-v_n$. We get $V(G')=n$ and $\Delta(G')=|V(G')|-1$. If $d(v_n) \geq 2$, then let $v_1 \in N(v_0)$ such that $d(v_1)=1$. Set $G'=G-v_1$.  We have $V(G')=n$ and $\Delta(G')  \geq |V(G')|-2$. In either case, we have $R(G')  >  \sqrt{n-1}+\tfrac{2(k+1)}{n\sqrt{n-1}}$ by the inductive hypothesis. The inductive step then follows from  Lemma \ref{LeInduBas}.

\vspace{0.1cm}
\noindent
{\bf Case 3:} $\Delta(G)=|V(G)|-3=n-2$.  Assume $\{v_{n-1},v_n\}=V(G) \setminus (\{v_0\} \cup N(v_0))$.  If one of $v_{n-1}$ and $v_n$ has degree one, say $v_n$,  then we let $G'=G-v_n$. We observe
$V(G')=n$ and $\Delta(G')=|V(G')|-2$.   If $d(v_{n-1}),d(v_n) \geq 2$, then let $v_1 \in N(v_0)$ such that $d(v_1)=1$. Set $G'=G-v_1$.  We have $V(G')=n$ and $\Delta(G') \geq |V(G')|-3$. In either case,  the inductive hypothesis gives  $R(G') >  \sqrt{n-1}+\tfrac{2(k+1)}{n\sqrt{n-1}}$. We  can prove the inductive step by using  Lemma \ref{LeInduBas}.
\end{proof}
The combination of Lemma \ref{general}  and Lemma \ref{LeBasic} yields the next lemma.
\begin{lemma} \label{induction}
Let $G$ be a connected graph with $n \geq 13$ vertices. If either of the following holds:
\begin{enumerate}
\item $\Delta(G)=n-1$ and $e(G)=n+k$ for $1\leq k\leq  10$;
\item $\Delta(G)=n-2$ and $e(G)=n+k$ for $ 5 \leq k \leq  10$;
\item $\Delta(G)=n-3$, and $e(G)=n+k$ for $8 \leq k \leq 10$,
\end{enumerate}
then we have
\[
\frac{q(G)}{R(G)} < \frac{n}{\sqrt{n-1}}.
\]
\end{lemma}

We are now ready to prove the main theorem.

\noindent
{\bf Proof of Theorem \ref{main}.} If $e(G)=n-1$, then $G$ is a tree.  We have  $q(G)=n$ and $R(G)\geq \sqrt{n-1}$ by Theorem \ref{ThBE}.
Thus,  $\tfrac{q(G)}{R(G)}\leq\frac{n}{\sqrt{n-1}}$. If $e(G)=n$, then  Theorem \ref{ThGL} implies $R(G)\geq R(S^{*}_n)=\tfrac{n-3}{\sqrt{n-1}}+\sqrt{\tfrac{2}{n-1}}+\tfrac{1}{2} > \sqrt{n-1}+\tfrac{2}{n\sqrt{n-1}}$ when $n\geq 12$.  By Lemma \ref{LeBasic}, we have  $\tfrac{q(G)}{R(G)} < \tfrac{n}{\sqrt{n-1}}$. For $n=12$,  we note $\tfrac{12}{\sqrt{11}}<\tfrac{11}{3}$.   For the rest of the proof,  we  assume  $e(G)=n+k$ with $k\geq 1$.  We first prove the second part of the theorem, namely, $n \geq 13$. We shall consider the following three cases depending on the range of $e(G)$.

\vspace{0.1cm}
\noindent
{\bf Case 1:} $e(G) \geq 2n^{3/2}$.  We get $\tfrac{q(G)}{R(G)}<\tfrac{n}{\sqrt{n-1}}$ by  Lemma \ref{dense}.

\vspace{0.1cm}
\noindent
{\bf Case 2:} $n+11 \leq e(G) \leq  \min\{2n^{3/2},\binom{n}{2}\}$. In this case,  $\tfrac{q(G)}{R(G)}<\tfrac{n}{\sqrt{n-1}}$ is given by Lemma \ref{1314} and Lemma \ref{reduce}.

\vspace{0.1cm}
\noindent
{\bf Case 3:} $n+1 \leq e(G) \leq n+10$.   We consider the following subcases depending on $\Delta(G)$.  We claim  $\tfrac{q(G)}{R(G)}<\tfrac{n}{\sqrt{n-1}}$ for each subcase.

\vspace{0.1cm}
{\bf Subcase 3.1:}  $\Delta(G)=n-1$.  Part 1 of Lemma \ref{induction} proves the claim.

\vspace{0.1cm}
{\bf Subcase 3.2:}  $\Delta(G)=n-2$.   The case of  $n+1 \leq e(G) \leq n+4$ is proved by  Part 3 of Lemma \ref{largedegree} and the case of  $n+5 \leq e(G) \leq n+10$ is proved by  Part 2 of Lemma \ref{induction}.

\vspace{0.1cm}
{\bf Subcase 3.3:}  $\Delta(G)=n-3$.  Part 2 of Lemma \ref{largedegree} proves the case of $n+1 \leq e(G) \leq n+7$ and Part 3 of Lemma \ref{induction} prove the case of  $n+8 \leq e(G) \leq n+10$.

\vspace{0.1cm}
{\bf Subcase 3.4:}  $n/2 \leq \Delta(G) \leq n-4$. Part 1 of Lemma \ref{largedegree} implies the claim.

\vspace{0.1cm}
{\bf Subcase 3.5:} $\Delta(G) < n/2$.   Lemma \ref{small}  gives us the claim.

From the argument for $ e(G) \geq n$ and $n \geq 13$, we get $\tfrac{q(G)}{R(G)}<\tfrac{n}{\sqrt{n-1}}$ when $e(G) \geq n$. Therefore,   $\tfrac{q(G)}{R(G)}=\tfrac{n}{\sqrt{n-1}}$ can only occur  for $e(G)=n-1$.  By Theorem \ref{ThBE}, we get the equality holds if and only if $G$ is a star when $n \geq 13$.

We are left with the case where $n=12$.  We shall use the function $g(m)$ from the proof of Theorem 10 in \cite{DengBalAyy}. Specialized to $n=12$, we get
 \[
 g(m)=\frac{\left(\tfrac{2m}{11}+10\right)\sqrt{2m-11}}{m}.
 \]
Let $m=e(G)$. With the help of computer, we get $g(m) < \tfrac{11}{3}$ for $21 \leq m \leq  \binom{12}{2}-1=65$ and $g(66) =\tfrac{11}{3}$. Equivalently, $\tfrac{q(G)}{R(G)}<\tfrac{11}{3}$ when $21 \leq m \leq 65$. We need only to prove the case of   $m=12+k$ for $1 \leq k \leq 8$.  Recall Lemmas \ref{LeBasic}, \ref{small}, \ref{n121}, \ref{n122}. Repeating the case analysis above, we  can show   $\tfrac{q(G)}{R(G)}< \tfrac{12}{\sqrt{11}}<\tfrac{11}{3}$ when $13 \leq e(G) \leq 20$. We already proved  $\tfrac{q(G)}{R(G)}<\tfrac{11}{3}$ when $e(G) \in \{n-1,n\}$. Therefore,  $\tfrac{q(G)}{R(G)}=\tfrac{11}{3}$ may hold only for $e(G)=66$, which turns out to be true because $G=K_{12}$.

We have  completed the proof of the main theorem.

\end{document}